\documentclass[12pt]{article}
\usepackage{amsmath,pslatex,amssymb,amsfonts,latexsym,amstext,amsmath,amsthm}
\begin{document}
\newcommand{\B}{\mathcal{B}}
\newcommand{\Cset}{\mathcal{D}}
\newtheorem{theorem}{Theorem}
\newtheorem{lemma}[theorem]{Lemma}
\newtheorem{corollary}[theorem]{Corollary}
\newtheorem{conjecture}[theorem]{Conjecture}
\newtheorem{example}[theorem]{Example}
\newtheorem{remark}[theorem]{Remark}
\newtheorem{definition}[theorem]{Definition}
\newtheorem{notation}[theorem]{Notation}
\newtheorem{summary}[theorem]{Summary}
\newtheorem{proposition}[theorem]{Proposition}

\newcommand\Zp{\makebox[.75em]{$1$}}
\newcommand\Zm{\makebox[.75em]{$-$}}
\newcommand\Zz{\makebox[.75em]{$0$}}
\newcommand\Zi{\makebox[.75em]{$i$}}
\newcommand\Zj{\makebox[.75em]{$j$}}

\title{Unbiased complex Hadamard matrices and bases}

\renewcommand{\thefootnote}{\fnsymbol{footnote}}

\author{Darcy Best and Hadi Kharaghani\footnotemark\\
Department of Mathematics \& Computer Science\\
University of Lethbridge\\
Lethbridge, Alberta, T1K 3M4\\
Canada\\
email: darcy.best@uleth.ca, kharaghani@uleth.ca }

\maketitle

\footnotetext[1]{Supported by an NSERC
Discovery Grant - Group.}

\begin{abstract}
We introduce mutually unbiased complex Hadamard ({\bf MUCH})
matrices  and show that the number of {\bf MUCH} matrices  of
order $2n$, $n$ odd, is at most $2$ and the  bound is attained for
$n=1,5,9$. Furthermore, we prove that certain pairs of mutually
unbiased complex Hadamard matrices of order $m$ can be used to
construct  pairs of unbiased real Hadamard matrices of
order $2m$. As a consequence we generate a new pair of unbiased
real Hadamard matrices of order $36$.
\end{abstract}

\noindent {\bf AMS Subject Classification:} Primary 05B20.

\noindent {\bf Keywords:}
complex Hadamard matrices, real Hadamard matrices, unbiased
Hadamard matrices, unbiased bases

\section{Preliminaries}

A {\it complex Hadamard} matrix is a matrix $H$ of order $n$ with
entries in $\{-1,1,i,-i\}$ and orthogonal rows in the usual
complex inner product on ${\mathbb{C}}^n$. If the entries of the
matrix consist of only $\pm 1$, we call the matrix a real Hadamard
matrix or a Hadamard matrix for short. Our main references for
complex and real Hadamard matrices are  \cite {ys,t}. Two complex
Hadamard matrices $H$ and $K$ of order $2n$ are called {\it
unbiased} if $HK^{*}=L$, where $K^*$ denotes the Hermitian
transpose of $K$ and all the entries of the matrix $L$ are of the
absolute value $\sqrt {2n}$. In this case, it follows that
$2n=a^2+b^2$, where $a$, $b$ are nonnegative integers. While there
has been a lot of interest in the class of mutually unbiased
unimodular complex Hadamard matrices, where the entries of
the matrices consist of unimodular complex numbers, see \cite
{cs,gr,wb} for details, it is only recently that some interest has been shown in
the existence and applications of mutually unbiased real Hadamard
matrices, see \cite {lwo}. Our aim in this paper is to concentrate
on matrices of order $2n$, $n$ odd, with entries in
$\{-1,1,i,-i\}$. We will find an upper bound for the number of
mutually unbiased complex Hadamard matrices of order $2n$, $n$
odd, denoted $|{\bf MUCH}(2n)|$, in the next section. We also
 report on the outcome of a computer search for maximal
classes of {\bf MUCH} matrices of orders $10$ and $18$. Section 3
is devoted to the study of unbiased real Hadamard matrices. We will
briefly discuss mutually unbiased bases in the last section. In
the presentation of matrices we use $j$ to denote $-i$ and $-$ to
denote $-1$.

\section{Unbiased complex Hadamard matrices}

Dealing with complex  matrices, i.e. matrices with entries  in
$\{-1,1,i,-i\}$, is quite different from working
with the unimodular complex matrices as the powerful character
theory is no longer applicable.
We begin this section with a well known, but important property
of complex Hadamard matrices.

\begin{lemma}\label{rregular}
Let $H=[h_{ij}]$ be a complex Hadamard matrix of order $n$
for which the absolute value of the row sums are all identical
and equal to {\bf r}. Then ${\bf r}=\sqrt n$.
\end{lemma}

\begin{proof}
For ${\bf e}$ being the all ones vector, we have
 $(H{\bf e})^*(H{\bf e}) = {\bf e}^* H^* H {\bf e} = {\bf e}^*nI{\bf e} =  n {\bf e}^*{\bf e} = n^2$. So, $\sum_{i=1}^n|r_i|^2$ $=n^2$, where
$r_i=\sum_{j=1}^nh_{ij}$, $1\le i\le n$. It follows that ${\bf
r}=\sqrt n$.
\end{proof}

A complex Hadamard matrix of order $n$ for which the absolute value of the row
sums are all equal to $\sqrt n$ is called {\it row regular}. It follows
from Lemma \ref{rregular} that for a row regular complex Hadamard
matrix $H=[h_{kj}]$ of order $2n$, $n$ odd, if $\sum_{j=1}^{2n}
h_{kj}=a+ib$, for some $k$, $1\le k \le 2n$, then $a^2+b^2=2n$ and so both $|a|$
and $|b|$ are odd integers.

\begin{lemma}\label{main3}
There is no pair of unbiased row regular complex Hadamard matrices
of order $2n$, $n$ odd.
\end{lemma}

\begin{proof}
Suppose on the contrary that there is a pair of row regular
complex Hadamard matrices $H$ and $K$ of order $2n$ such that
$HK^{*}=L$, where the entries of $L$ are of absolute value $\sqrt
{2n}$. Let $J$ be the matrix of all one entries of order $2n$.
Then the matrix
$${\frac{1}{1+i}}(H+J)\,{\frac{1}{1+i}}(K^*+J)$$
is a complex
integer matrix (i.e. all entries of the matrix consist of Gaussian
integers). To see this note that the entries of both matrices
$\frac{1}{1+i}(H+J)$ and $\frac{1}{1+i}(K^*+J)$
belong to the set $\{0,1, -i,1-i\}$. Observing
that
$$\frac{1}{1+i}(H+J)\,\frac{1}{1+i}(K^*+J)=
\frac{-i}{2}(HK^*+HJ+JK^*+2nJ)$$
and that all the entries of the matrices $HK^*$, $HJ$ and $JK^*$
consist of numbers of the form $x+iy$, where both $|x|$ and
$|y|$ are odd integers, we get a contradiction.
\end{proof}

Note that in the above proof we only use the fact
that all the entries of the matrices $HK^*$, $HJ$ and $JK^*$
consist of numbers of the form $x+iy$, where both $|x|$ and
$|y|$ are odd integers. So if there are
two complex Hadamard matrices $H$, $K$ of order $2n$,
$n$ odd, for which the row sums of $H$ and $K$
are all of the form $x+iy$, where both $|x|$ and
$|y|$ are odd integers, then none of the entries of $HK^*$ are of this form.
Consequently, such $H,K$ can not be unbiased, for the entries of  $HK^*=[H_{ij}]$,
$|L_{ij}|^2 = 2n$, which must be a sum of two odd squares.

\begin{theorem}\label{theorem:numberMUCHtwo}
For any odd integer $n$, $|{\bf MUCH}(2n)|\le 2$.
\end{theorem}

\begin{proof}
Suppose on the contrary that there are more than two {\bf MUCH}
matrices of order $2n$. By multiplying the columns of all matrices
by appropriate numbers we can make the first row of one of the
matrices to be all equal to one. The new matrices form a set of
{\bf MUCH} matrices which contain at least two row regular
Hadamard matrices of order $2n$, contradicting Lemma \ref{rregular}
and thus the result follows.
\end{proof}

\begin{example}
Let $$
H = \left(\begin{array}{cc}
1 & 1\\
1 & -
\end{array}\right), \quad
K = \left(\begin{array}{cc}
1 & i\\
i & 1
\end{array}\right).
$$
Then $$HK^*=\left(\begin{array}{cc}
1-i & 1-i\\
i+1 & -i-1
\end{array}\right).$$
This shows the inequality in the Theorem \ref{theorem:numberMUCHtwo}
is sharp for $n=1$.
\end{example}

We have  conducted a computer search and found many maximal sets
of {\bf MUCH} matrices of orders $10$ and $18$. One representative from
each of these pairs of matrices is listed below in Tables \ref{matC10}
and \ref{matC18}.

\begin{table}[htbp!]
\centering
\caption{A pair $H$, $K$ of unbiased complex Hadamard matrices of order 10}
\label{matC10}
$$
\left(
\begin{array}{c}
 \Zp\Zp\Zp\Zp\Zp\Zp\Zp\Zp\Zp\Zp\\
 \Zp\Zp\Zm\Zm\Zm\Zi\Zj\Zp\Zm\Zp\\
 \Zp\Zm\Zp\Zm\Zm\Zj\Zm\Zi\Zp\Zp\\
 \Zp\Zm\Zm\Zp\Zm\Zj\Zp\Zp\Zi\Zm\\
 \Zp\Zm\Zm\Zm\Zp\Zi\Zp\Zm\Zp\Zj\\
 \Zp\Zj\Zi\Zi\Zj\Zm\Zi\Zj\Zj\Zi\\
 \Zp\Zi\Zm\Zp\Zp\Zj\Zm\Zm\Zm\Zp\\
 \Zp\Zp\Zj\Zp\Zm\Zi\Zm\Zm\Zp\Zm\\
 \Zp\Zm\Zp\Zj\Zp\Zi\Zm\Zp\Zm\Zm\\
 \Zp\Zp\Zp\Zm\Zi\Zj\Zp\Zm\Zm\Zm
\end{array}
\right), \qquad
\left(
\begin{array}{c}
 \Zj\Zm\Zm\Zp\Zp\Zp\Zp\Zj\Zi\Zp\\
 \Zj\Zp\Zp\Zp\Zi\Zm\Zj\Zp\Zi\Zj\\
 \Zp\Zj\Zi\Zj\Zp\Zm\Zp\Zi\Zp\Zi\\
 \Zp\Zi\Zj\Zp\Zj\Zj\Zm\Zp\Zp\Zi\\
 \Zi\Zp\Zm\Zi\Zj\Zi\Zp\Zp\Zp\Zj\\
 \Zi\Zp\Zi\Zj\Zp\Zp\Zm\Zp\Zm\Zp\\
 \Zp\Zi\Zp\Zp\Zp\Zi\Zi\Zm\Zj\Zj\\
 \Zm\Zj\Zp\Zp\Zm\Zp\Zi\Zi\Zp\Zp\\
 \Zp\Zp\Zj\Zm\Zi\Zp\Zj\Zm\Zp\Zp\\
 \Zm\Zp\Zp\Zi\Zp\Zj\Zp\Zj\Zj\Zi
\end{array}
\right)
$$
\end{table}

\begin{table}[t!]
\centering
\caption{A pair $H$, $K$ of unbiased complex Hadamard matrices of order 18}
\label{matC18}
$$
\left(
\begin{array}{c}
 \Zp\Zp\Zp\Zp\Zp\Zp\Zp\Zp\Zp\Zp\Zp\Zp\Zp\Zp\Zp\Zp\Zp\Zp\\
 \Zp\Zm\Zj\Zj\Zi\Zj\Zi\Zj\Zj\Zj\Zj\Zi\Zi\Zi\Zj\Zi\Zi\Zi\\
 \Zp\Zj\Zm\Zj\Zi\Zi\Zi\Zi\Zj\Zi\Zj\Zi\Zj\Zj\Zi\Zi\Zj\Zj\\
 \Zp\Zj\Zj\Zm\Zi\Zi\Zj\Zi\Zi\Zj\Zi\Zi\Zj\Zi\Zj\Zj\Zj\Zi\\
 \Zp\Zi\Zi\Zi\Zm\Zj\Zi\Zi\Zj\Zj\Zi\Zi\Zi\Zj\Zj\Zj\Zj\Zj\\
 \Zp\Zj\Zi\Zi\Zj\Zm\Zi\Zi\Zi\Zj\Zj\Zj\Zj\Zj\Zj\Zi\Zi\Zi\\
 \Zp\Zi\Zi\Zj\Zi\Zi\Zm\Zj\Zj\Zj\Zi\Zj\Zj\Zj\Zi\Zj\Zi\Zi\\
 \Zp\Zj\Zi\Zi\Zi\Zi\Zj\Zm\Zj\Zi\Zj\Zj\Zi\Zi\Zj\Zj\Zi\Zj\\
 \Zp\Zj\Zj\Zi\Zj\Zi\Zj\Zj\Zm\Zj\Zi\Zi\Zi\Zj\Zi\Zi\Zi\Zj\\
 \Zp\Zj\Zi\Zj\Zj\Zj\Zj\Zi\Zj\Zm\Zi\Zj\Zi\Zi\Zi\Zi\Zj\Zi\\
 \Zp\Zj\Zj\Zi\Zi\Zj\Zi\Zj\Zi\Zi\Zm\Zj\Zi\Zj\Zi\Zj\Zj\Zi\\
 \Zp\Zi\Zi\Zi\Zi\Zj\Zj\Zj\Zi\Zj\Zj\Zm\Zj\Zi\Zi\Zi\Zj\Zj\\
 \Zp\Zi\Zj\Zj\Zi\Zj\Zj\Zi\Zi\Zi\Zi\Zj\Zm\Zj\Zj\Zi\Zi\Zj\\
 \Zp\Zi\Zj\Zi\Zj\Zj\Zj\Zi\Zj\Zi\Zj\Zi\Zj\Zm\Zi\Zj\Zi\Zi\\
 \Zp\Zj\Zi\Zj\Zj\Zj\Zi\Zj\Zi\Zi\Zi\Zi\Zj\Zi\Zm\Zj\Zi\Zj\\
 \Zp\Zi\Zi\Zj\Zj\Zi\Zj\Zj\Zi\Zi\Zj\Zi\Zi\Zj\Zj\Zm\Zj\Zi\\
 \Zp\Zi\Zj\Zj\Zj\Zi\Zi\Zi\Zi\Zj\Zj\Zj\Zi\Zi\Zi\Zj\Zm\Zj\\
 \Zp\Zi\Zj\Zi\Zj\Zi\Zi\Zj\Zj\Zi\Zi\Zj\Zj\Zi\Zj\Zi\Zj\Zm
\end{array}
\right),\quad
\left(
\begin{array}{c}
 \Zp\Zi\Zi\Zj\Zj\Zj\Zj\Zi\Zi\Zj\Zm\Zp\Zj\Zp\Zp\Zj\Zp\Zm\\
 \Zp\Zi\Zi\Zj\Zi\Zj\Zi\Zj\Zm\Zi\Zi\Zp\Zp\Zm\Zp\Zi\Zj\Zp\\
 \Zp\Zi\Zi\Zi\Zi\Zi\Zp\Zp\Zj\Zm\Zj\Zi\Zj\Zp\Zj\Zp\Zm\Zi\\
 \Zp\Zj\Zi\Zi\Zi\Zp\Zj\Zi\Zi\Zp\Zp\Zp\Zi\Zj\Zm\Zi\Zj\Zm\\
 \Zp\Zj\Zi\Zm\Zj\Zj\Zp\Zj\Zj\Zp\Zp\Zm\Zi\Zj\Zp\Zj\Zi\Zi\\
 \Zp\Zi\Zj\Zm\Zm\Zi\Zj\Zj\Zp\Zi\Zj\Zj\Zp\Zj\Zi\Zp\Zp\Zj\\
 \Zp\Zi\Zj\Zp\Zp\Zi\Zi\Zm\Zi\Zj\Zi\Zm\Zi\Zj\Zj\Zp\Zp\Zi\\
 \Zj\Zm\Zp\Zi\Zj\Zp\Zi\Zi\Zp\Zi\Zi\Zi\Zj\Zj\Zp\Zi\Zi\Zj\\
 \Zi\Zp\Zp\Zj\Zi\Zp\Zi\Zi\Zj\Zj\Zj\Zj\Zp\Zj\Zj\Zm\Zi\Zj\\
 \Zp\Zj\Zp\Zj\Zj\Zi\Zm\Zp\Zj\Zp\Zm\Zi\Zi\Zi\Zj\Zp\Zj\Zj\\
 \Zi\Zm\Zj\Zi\Zp\Zj\Zj\Zi\Zj\Zi\Zi\Zj\Zp\Zp\Zj\Zj\Zj\Zp\\
 \Zp\Zj\Zp\Zp\Zm\Zi\Zi\Zi\Zm\Zi\Zp\Zj\Zj\Zi\Zi\Zj\Zp\Zi\\
 \Zj\Zp\Zj\Zj\Zi\Zm\Zj\Zi\Zj\Zp\Zi\Zi\Zp\Zp\Zi\Zi\Zi\Zi\\
 \Zp\Zj\Zm\Zp\Zm\Zp\Zj\Zj\Zp\Zj\Zi\Zj\Zj\Zi\Zj\Zi\Zi\Zp\\
 \Zp\Zp\Zj\Zi\Zj\Zj\Zi\Zm\Zj\Zj\Zj\Zp\Zm\Zp\Zi\Zi\Zj\Zp\\
 \Zp\Zp\Zj\Zi\Zi\Zj\Zm\Zp\Zi\Zm\Zp\Zi\Zi\Zi\Zp\Zj\Zi\Zj\\
 \Zi\Zj\Zp\Zp\Zi\Zm\Zp\Zj\Zp\Zj\Zi\Zi\Zj\Zj\Zi\Zj\Zj\Zj\\
 \Zj\Zi\Zi\Zi\Zp\Zi\Zp\Zp\Zj\Zj\Zi\Zj\Zi\Zi\Zi\Zm\Zp\Zj
\end{array}
\right)
$$
\end{table}

We believe that the upper bound in Theorem
\ref{theorem:numberMUCHtwo} is sharp for every odd integer  $n$ for
which $2n$ is the order of a row regular complex Hadamard matrix.
The following conjecture includes this and a conjecture regarding
the existence of row regular complex Hadamard matrices.

\begin{conjecture}
$|{\bf MUCH}(2n)|=2$ for all odd integers $n$, where $2n$ is a
sum of two squares.
\end{conjecture}

The existence of row regular Hadamard matrix is a necessary condition to have
two {\bf MUCH}'s (see the proof of Theorem 3). For matrices of size $2n$, $n$ odd,
the existence of a row regular Hadamard matrix is, in turn, conditioned by
existence of integers $a,b$ such that $2n = a^2 + b^2$ (see lemma 1).

\section{ Unbiased real Hadamard matrices}

Two Hadamard matrices $H$, $K$ of order $n$ are called unbiased,
if $HK^t=L$, where the absolute values of all entries of $L$ are
equal to $\sqrt n$. It follows that $L=\sqrt nA$, where $A$ is a
Hadamard matrix of order $n$. It is only recently that interest has
been shown
in unbiased
Hadamard matrices
\cite{bstw,wb} and some new applications have emerged \cite{lwo}.
Pairs of unbiased Hadamard matrices exist only in square orders,
as $L = HK^t$, with moduli of entries of $L$ equal to $\sqrt n$, is a matrix of integers.
It is known and easy to prove (as shown below) that the maximum
number of mutually unbiased Hadamard matrices of order $4n^2$, $n$
odd, does not exceed 2. Although Lemma 13 provides an upper bound for the number of what we
call weakly unbiased Hadamard matrices (see Definition. 7), unbiased Hadamard matrices of order
$4n$, $n$ an odd square, belong to this class (see Remark. 8). So Lemma 13
also applies to them. Until very recently no example for which
the upper bound $2$ is attained was known besides the trivial
example of Hadamard matrices of order $4$. The first non-trivial
example of unbiased Hadamard matrices of order $36$ is shown in
\cite {hko}. The approach in \cite{hko} was to use a database of
known Hadamard matrices of order $36$ to search for matrices with
unbiased mates. Interestingly, only a very small fraction of the
over 3 million known matrices of order $36$ which were tested  had
unbiased mates. In this section we show that some sets of {\bf
MUCH} matrices of order $2n$ can be used to generate sets of
mutually unbiased Hadamard matrices of order $4n$. Having found
pairs of {\bf MUCH} of order $18$, we have many pairs of mutually
unbiased Hadamard matrices of order $36$. We begin with a known
\cite {bstw} and simple lemma. Our motivation for including the
proof here will follow.

\begin{lemma}\label{r-unbiased}
There is no pair of unbiased row regular Hadamard matrices of
order $4n^2$, $n$ odd.
\end{lemma}

\begin{proof}
Repeating the line of proof of Lemma \ref{main3}, we have
$$\frac{1}{2}(H+J)\frac{1}{2}(K^t+J)=\frac{1}{4}(HK^t+HJ+JK^t+4n^2J).$$
Noting that $HK^t=2nL$, where $L$ is a Hadamard matrix, we get a
contradiction to the fact that the left side of the above identity
is an integer matrix.
\end{proof}

A quick glance at the above proof reveals that
$HJ+JK^t+4n^2J\equiv 0 \pmod{4}$, if and only if $HJ+JK^t\equiv 0
\pmod{4}$. Assuming that $HJ+JK^t\equiv 0 \pmod{4}$, we get a
contradiction if we assume one (or equivalently all) of the entries of
$HK^t$ is equal to $2\pmod{4}$. This is our motivation for the
following definition.

\begin{definition}
Two Hadamard matrices $H$, $K$ of order $n$ are said to be weakly
unbiased, if  $ |\{|a_{ij}|: 1\le i\le n, 1\le j\le n \}|\le 2$,
and $HK^t=[a_{ij}]\equiv 2J\pmod{4}$.
\end{definition}

\begin{remark}
Note that for the unbiased Hadamard matrices $H$, $K$ of order $n$,
$ |\{|a_{ij}|: 1\le i\le n, 1\le j\le n \}|=1$, where
$HK^t=[a_{ij}]$. So weakly unbiased Hadamard matrices are the
natural extension of unbiased Hadamard matrices of order $4n$, $n$ an odd square.
\end{remark}

The following lemma is immediate,
using equality from the proof of Lemma \ref{r-unbiased}.

\begin{lemma}\label{lemma:no2mod4}
Let $H$, $K$ be Hadamard matrices of order $4n$ such that
$HJ+JK^t\equiv 0 \pmod{4}$. Then no entry of $HK^t$ is equal to
$2\pmod{4}$.
\end{lemma}

\begin{definition}
Two Hadamard matrices $H$, $K$ of the same order are called to  be
modularly homogeneous if $HJ+JK^t\equiv 0 \pmod{4}$.
\end{definition}

\begin{lemma}\label{mhh}
There is no pair $H$, $K$ of modularly homogeneous Hadamard  matrices
of order $4n$ for which $HK^t\equiv 2J\pmod{4}$.
\end{lemma}

\begin{proof}
This follows from Lemma \ref{lemma:no2mod4}.
\end{proof}

\begin{remark}
The assumption that $H$ and $K$ are modularly homogeneous in
Lemma \ref{mhh} is essential.
The Hadamard matrices of order 12 in Table \ref{matw12}
are weakly unbiased, that is $HK^t\equiv 2J\pmod{4}$,
but not modularly homogeneous.
It is noteworthy that the number of entries with value $2$ or $6$
in $HK^t$ is not balanced as there are more $2$ entries than
$6$ entries.
\end{remark}

\begin{table}[htbp!]
\centering
\caption{A pair $H$, $K$ of weakly unbiased Hadamard matrices of order 12}
\label{matw12}
$$
\left(
\begin{array}{c}
\Zp\Zp\Zp\Zm\Zp\Zp\Zm\Zp\Zp\Zm\Zp\Zp\\
\Zp\Zp\Zp\Zp\Zm\Zp\Zp\Zm\Zp\Zp\Zm\Zp\\
\Zp\Zp\Zp\Zp\Zp\Zm\Zp\Zp\Zm\Zp\Zp\Zm\\
\Zp\Zm\Zm\Zp\Zp\Zp\Zm\Zp\Zp\Zp\Zm\Zm\\
\Zm\Zp\Zm\Zp\Zp\Zp\Zp\Zm\Zp\Zm\Zp\Zm\\
\Zm\Zm\Zp\Zp\Zp\Zp\Zp\Zp\Zm\Zm\Zm\Zp\\
\Zp\Zm\Zm\Zp\Zm\Zm\Zp\Zp\Zp\Zm\Zp\Zp\\
\Zm\Zp\Zm\Zm\Zp\Zm\Zp\Zp\Zp\Zp\Zm\Zp\\
\Zm\Zm\Zp\Zm\Zm\Zp\Zp\Zp\Zp\Zp\Zp\Zm\\
\Zp\Zm\Zm\Zm\Zp\Zp\Zp\Zm\Zm\Zp\Zp\Zp\\
\Zm\Zp\Zm\Zp\Zm\Zp\Zm\Zp\Zm\Zp\Zp\Zp\\
\Zm\Zm\Zp\Zp\Zp\Zm\Zm\Zm\Zp\Zp\Zp\Zp
\end{array}
\right), \qquad
\left(
\begin{array}{c}
\Zm\Zp\Zp\Zm\Zp\Zp\Zp\Zm\Zm\Zp\Zp\Zp\\
\Zp\Zm\Zp\Zp\Zm\Zp\Zm\Zp\Zm\Zp\Zp\Zp\\
\Zp\Zp\Zm\Zp\Zp\Zm\Zm\Zm\Zp\Zp\Zp\Zp\\
\Zp\Zm\Zm\Zm\Zp\Zp\Zp\Zp\Zp\Zm\Zp\Zp\\
\Zm\Zp\Zm\Zp\Zm\Zp\Zp\Zp\Zp\Zp\Zm\Zp\\
\Zm\Zm\Zp\Zp\Zp\Zm\Zp\Zp\Zp\Zp\Zp\Zm\\
\Zm\Zp\Zp\Zm\Zm\Zm\Zm\Zp\Zp\Zm\Zp\Zp\\
\Zp\Zm\Zp\Zm\Zm\Zm\Zp\Zm\Zp\Zp\Zm\Zp\\
\Zp\Zp\Zm\Zm\Zm\Zm\Zp\Zp\Zm\Zp\Zp\Zm\\
\Zm\Zm\Zm\Zp\Zm\Zm\Zp\Zm\Zm\Zm\Zp\Zp\\
\Zm\Zm\Zm\Zm\Zp\Zm\Zm\Zp\Zm\Zp\Zm\Zp\\
\Zm\Zm\Zm\Zm\Zm\Zp\Zm\Zm\Zp\Zp\Zp\Zm
\end{array}
\right)
$$
\end{table}

\begin{lemma}\label{w(n)}
Let $w(n)$ be the number of mutually weakly unbiased Hadamard
matrices of order $4n$, $n$ odd, then $w(n)\le 2$.
\end{lemma}

\begin{proof}
Suppose on the contrary that there are more than two mutually weakly
unbiased Hadamard matrices of order $4n$. By negating the
appropriate columns of all matrices, we may assume that one of the
matrices has one normalized row. Select two other matrices, say
$H$, $K$. Then $HJ+JK^t\equiv 0 \pmod{4}$ and $HK^t\equiv 2J
\pmod{4}$, contradicting Lemma \ref{lemma:no2mod4}.
\end{proof}

We are now ready for the main result of this section and our
reason for studying unbiased complex Hadamard matrices. We need to
introduce a notation first. For the integers $a,b$ let
$G(a,b)=\{a\pm ib,-a\pm ib,ia\pm b,-ia\pm b\}$.

\begin{theorem}\label{theorem:existenceWUHM}
Let $H$, $K$ be a pair of unbiased complex Hadamard matrices of order
$2n$, $n$ odd, for which the entries of $HK^*$ are all in $G(a,b)$,
where $2n=a^2+b^2$, $a,b$ odd integers. Then there is a pair of
weakly unbiased Hadamard matrices of order $4n$.
\end{theorem}

\begin{proof}
Let $H=A+iB$, $K=C+iD$, where $A,B$ and $C,D$ are
$(0,\pm 1)$-matrices of order $2n$
such that $A\pm B$ and $C\pm D$ are $\pm 1$-matrices.
Consider the matrices
$$H'= \left(\begin{array}{cc}
1 & 1\\
1 & -
\end{array}\right)\otimes A + \left(\begin{array}{cc}
- & 1\\
1 & 1
\end{array}\right)\otimes B$$
and $$K'=\left(\begin{array}{cc}
1 & 1\\
1 & -
\end{array}\right)\otimes C + \left(\begin{array}{cc}
-& 1\\
1 & 1
\end{array}\right)\otimes D.$$
It is only a routine calculation to see that $H',K'$ are Hadamard
matrices of order $4n$. Let $HK^*=E+iF$, where $E,F$ are $(\pm
a,\pm b)$-matrices of order $2n$. We have
$$H'K'^t=\left(\begin{array}{cc}
2(AC^t+BD^t) & -2(BC^t-AD^t)\\
2(BC^t-AD^t) & 2(AC^t+BD^t)
\end{array}\right)=\left(\begin{array}{cc}
2E & -2F\\
2F & 2E
\end{array}\right).$$
Using the fact that the entries of $HK^*$ are in $G(a,b)$
and noting that
$E,F$ are $(\pm a,\pm b)$-matrices, where $|a|$ ,$|b|$ are odd integers, it
follows that $H',K'$ are weakly unbiased.
\end{proof}

\begin{remark}
The spread of $a$'s and $b$'s in $H'K'^t$ is uniform; there are as
many $a$'s in $H'K'^t$ as $b$'s. We think the assumption that
all the entries of $HK^*$ belong to $G(a,b)$ is not necessary,
but we cannot prove it.\end{remark}

\begin{theorem}\label{main-2}
Let $H$, $K$ be a pair of unbiased complex Hadamard matrices of order
$2n$, where $n=a^2$, $a$ odd (and so $2n=a^2+a^2$)
for which the entries of $HK^*$ are in $G(a,a)$.
Then $H',K'$ constructed above form
a pair of unbiased Hadamard matrices of order $4n=4a^2$.
\end{theorem}

\begin{proof}
Note that in this case the matrices $E$ and $F$
in the proof of Theorem \ref{theorem:existenceWUHM}  are
both $\pm a$-matrices.
\end{proof}

\begin{corollary}\label{36}
There is a pair of unbiased Hadamard matrices of order $36$.
\end{corollary}

\begin{proof}
We apply Theorem \ref{main-2} to the pair of
unbiased complex Hadamard matrices of order 18 of Table \ref{matC18}.
The resulting pair of matrices is given in Tables \ref{mat36h} and
\ref{mat36k}.
The fact that all entries of $HK^*$
are in $G(a,a)$ is automatic in this case,
as 18 is sum of two squares in only one way.
\end{proof}

\begin{table}[htbp!]
\centering
\caption{A pair of unbiased Hadamard matrices of order 36: first matrix}
\label{mat36h}
$$
H=\left(
\begin{array}{c}
\Zp\Zp\Zp\Zp\Zp\Zp\Zp\Zp\Zp\Zp\Zp\Zp\Zp\Zp\Zp\Zp\Zp\Zp\Zp\Zp\Zp\Zp\Zp\Zp\Zp\Zp\Zp\Zp\Zp\Zp\Zp\Zp\Zp\Zp\Zp\Zp\\
\Zp\Zm\Zp\Zm\Zp\Zm\Zp\Zm\Zp\Zm\Zp\Zm\Zp\Zm\Zp\Zm\Zp\Zm\Zp\Zm\Zp\Zm\Zp\Zm\Zp\Zm\Zp\Zm\Zp\Zm\Zp\Zm\Zp\Zm\Zp\Zm\\
\Zp\Zp\Zm\Zm\Zm\Zp\Zm\Zp\Zp\Zm\Zm\Zp\Zp\Zm\Zp\Zm\Zp\Zm\Zm\Zp\Zm\Zp\Zp\Zm\Zp\Zm\Zp\Zm\Zm\Zp\Zp\Zm\Zm\Zp\Zm\Zp\\
\Zp\Zm\Zm\Zp\Zp\Zp\Zp\Zp\Zm\Zm\Zp\Zp\Zm\Zm\Zm\Zm\Zm\Zm\Zp\Zp\Zp\Zp\Zm\Zm\Zm\Zm\Zm\Zm\Zp\Zp\Zm\Zm\Zp\Zp\Zp\Zp\\
\Zp\Zp\Zm\Zp\Zm\Zm\Zm\Zp\Zm\Zp\Zp\Zm\Zm\Zp\Zp\Zm\Zp\Zm\Zp\Zm\Zm\Zp\Zm\Zp\Zp\Zm\Zp\Zm\Zp\Zm\Zm\Zp\Zp\Zm\Zm\Zp\\
\Zp\Zm\Zp\Zp\Zm\Zp\Zp\Zp\Zp\Zp\Zm\Zm\Zp\Zp\Zm\Zm\Zm\Zm\Zm\Zm\Zp\Zp\Zp\Zp\Zm\Zm\Zm\Zm\Zm\Zm\Zp\Zp\Zm\Zm\Zp\Zp\\
\Zp\Zp\Zm\Zp\Zm\Zp\Zm\Zm\Zm\Zp\Zm\Zp\Zp\Zm\Zm\Zp\Zp\Zm\Zp\Zm\Zp\Zm\Zm\Zp\Zm\Zp\Zp\Zm\Zp\Zm\Zp\Zm\Zm\Zp\Zp\Zm\\
\Zp\Zm\Zp\Zp\Zp\Zp\Zm\Zp\Zp\Zp\Zp\Zp\Zm\Zm\Zp\Zp\Zm\Zm\Zm\Zm\Zm\Zm\Zp\Zp\Zp\Zp\Zm\Zm\Zm\Zm\Zm\Zm\Zp\Zp\Zm\Zm\\
\Zp\Zp\Zp\Zm\Zm\Zp\Zm\Zp\Zm\Zm\Zm\Zp\Zm\Zp\Zp\Zm\Zm\Zp\Zp\Zm\Zp\Zm\Zp\Zm\Zm\Zp\Zm\Zp\Zp\Zm\Zp\Zm\Zp\Zm\Zm\Zp\\
\Zp\Zm\Zm\Zm\Zp\Zp\Zp\Zp\Zm\Zp\Zp\Zp\Zp\Zp\Zm\Zm\Zp\Zp\Zm\Zm\Zm\Zm\Zm\Zm\Zp\Zp\Zp\Zp\Zm\Zm\Zm\Zm\Zm\Zm\Zp\Zp\\
\Zp\Zp\Zm\Zp\Zp\Zm\Zm\Zp\Zm\Zp\Zm\Zm\Zm\Zp\Zm\Zp\Zp\Zm\Zm\Zp\Zp\Zm\Zp\Zm\Zp\Zm\Zm\Zp\Zm\Zp\Zp\Zm\Zp\Zm\Zp\Zm\\
\Zp\Zm\Zp\Zp\Zm\Zm\Zp\Zp\Zp\Zp\Zm\Zp\Zp\Zp\Zp\Zp\Zm\Zm\Zp\Zp\Zm\Zm\Zm\Zm\Zm\Zm\Zp\Zp\Zp\Zp\Zm\Zm\Zm\Zm\Zm\Zm\\
\Zp\Zp\Zp\Zm\Zm\Zp\Zp\Zm\Zm\Zp\Zm\Zp\Zm\Zm\Zm\Zp\Zm\Zp\Zp\Zm\Zm\Zp\Zp\Zm\Zp\Zm\Zp\Zm\Zm\Zp\Zm\Zp\Zp\Zm\Zp\Zm\\
\Zp\Zm\Zm\Zm\Zp\Zp\Zm\Zm\Zp\Zp\Zp\Zp\Zm\Zp\Zp\Zp\Zp\Zp\Zm\Zm\Zp\Zp\Zm\Zm\Zm\Zm\Zm\Zm\Zp\Zp\Zp\Zp\Zm\Zm\Zm\Zm\\
\Zp\Zp\Zp\Zm\Zp\Zm\Zm\Zp\Zp\Zm\Zm\Zp\Zm\Zp\Zm\Zm\Zm\Zp\Zm\Zp\Zp\Zm\Zm\Zp\Zp\Zm\Zp\Zm\Zp\Zm\Zm\Zp\Zm\Zp\Zp\Zm\\
\Zp\Zm\Zm\Zm\Zm\Zm\Zp\Zp\Zm\Zm\Zp\Zp\Zp\Zp\Zm\Zp\Zp\Zp\Zp\Zp\Zm\Zm\Zp\Zp\Zm\Zm\Zm\Zm\Zm\Zm\Zp\Zp\Zp\Zp\Zm\Zm\\
\Zp\Zp\Zp\Zm\Zp\Zm\Zp\Zm\Zm\Zp\Zp\Zm\Zm\Zp\Zm\Zp\Zm\Zm\Zm\Zp\Zm\Zp\Zp\Zm\Zm\Zp\Zp\Zm\Zp\Zm\Zp\Zm\Zm\Zp\Zm\Zp\\
\Zp\Zm\Zm\Zm\Zm\Zm\Zm\Zm\Zp\Zp\Zm\Zm\Zp\Zp\Zp\Zp\Zm\Zp\Zp\Zp\Zp\Zp\Zm\Zm\Zp\Zp\Zm\Zm\Zm\Zm\Zm\Zm\Zp\Zp\Zp\Zp\\
\Zp\Zp\Zm\Zp\Zp\Zm\Zp\Zm\Zp\Zm\Zm\Zp\Zp\Zm\Zm\Zp\Zm\Zp\Zm\Zm\Zm\Zp\Zm\Zp\Zp\Zm\Zm\Zp\Zp\Zm\Zp\Zm\Zp\Zm\Zm\Zp\\
\Zp\Zm\Zp\Zp\Zm\Zm\Zm\Zm\Zm\Zm\Zp\Zp\Zm\Zm\Zp\Zp\Zp\Zp\Zm\Zp\Zp\Zp\Zp\Zp\Zm\Zm\Zp\Zp\Zm\Zm\Zm\Zm\Zm\Zm\Zp\Zp\\
\Zp\Zp\Zm\Zp\Zm\Zp\Zp\Zm\Zp\Zm\Zp\Zm\Zm\Zp\Zp\Zm\Zm\Zp\Zm\Zp\Zm\Zm\Zm\Zp\Zm\Zp\Zp\Zm\Zm\Zp\Zp\Zm\Zp\Zm\Zp\Zm\\
\Zp\Zm\Zp\Zp\Zp\Zp\Zm\Zm\Zm\Zm\Zm\Zm\Zp\Zp\Zm\Zm\Zp\Zp\Zp\Zp\Zm\Zp\Zp\Zp\Zp\Zp\Zm\Zm\Zp\Zp\Zm\Zm\Zm\Zm\Zm\Zm\\
\Zp\Zp\Zp\Zm\Zm\Zp\Zm\Zp\Zp\Zm\Zp\Zm\Zp\Zm\Zm\Zp\Zp\Zm\Zm\Zp\Zm\Zp\Zm\Zm\Zm\Zp\Zm\Zp\Zp\Zm\Zm\Zp\Zp\Zm\Zp\Zm\\
\Zp\Zm\Zm\Zm\Zp\Zp\Zp\Zp\Zm\Zm\Zm\Zm\Zm\Zm\Zp\Zp\Zm\Zm\Zp\Zp\Zp\Zp\Zm\Zp\Zp\Zp\Zp\Zp\Zm\Zm\Zp\Zp\Zm\Zm\Zm\Zm\\
\Zp\Zp\Zp\Zm\Zp\Zm\Zm\Zp\Zm\Zp\Zp\Zm\Zp\Zm\Zp\Zm\Zm\Zp\Zp\Zm\Zm\Zp\Zm\Zp\Zm\Zm\Zm\Zp\Zm\Zp\Zp\Zm\Zm\Zp\Zp\Zm\\
\Zp\Zm\Zm\Zm\Zm\Zm\Zp\Zp\Zp\Zp\Zm\Zm\Zm\Zm\Zm\Zm\Zp\Zp\Zm\Zm\Zp\Zp\Zp\Zp\Zm\Zp\Zp\Zp\Zp\Zp\Zm\Zm\Zp\Zp\Zm\Zm\\
\Zp\Zp\Zp\Zm\Zp\Zm\Zp\Zm\Zm\Zp\Zm\Zp\Zp\Zm\Zp\Zm\Zp\Zm\Zm\Zp\Zp\Zm\Zm\Zp\Zm\Zp\Zm\Zm\Zm\Zp\Zm\Zp\Zp\Zm\Zm\Zp\\
\Zp\Zm\Zm\Zm\Zm\Zm\Zm\Zm\Zp\Zp\Zp\Zp\Zm\Zm\Zm\Zm\Zm\Zm\Zp\Zp\Zm\Zm\Zp\Zp\Zp\Zp\Zm\Zp\Zp\Zp\Zp\Zp\Zm\Zm\Zp\Zp\\
\Zp\Zp\Zm\Zp\Zp\Zm\Zp\Zm\Zp\Zm\Zm\Zp\Zm\Zp\Zp\Zm\Zp\Zm\Zp\Zm\Zm\Zp\Zp\Zm\Zm\Zp\Zm\Zp\Zm\Zm\Zm\Zp\Zm\Zp\Zp\Zm\\
\Zp\Zm\Zp\Zp\Zm\Zm\Zm\Zm\Zm\Zm\Zp\Zp\Zp\Zp\Zm\Zm\Zm\Zm\Zm\Zm\Zp\Zp\Zm\Zm\Zp\Zp\Zp\Zp\Zm\Zp\Zp\Zp\Zp\Zp\Zm\Zm\\
\Zp\Zp\Zp\Zm\Zm\Zp\Zp\Zm\Zp\Zm\Zp\Zm\Zm\Zp\Zm\Zp\Zp\Zm\Zp\Zm\Zp\Zm\Zm\Zp\Zp\Zm\Zm\Zp\Zm\Zp\Zm\Zm\Zm\Zp\Zm\Zp\\
\Zp\Zm\Zm\Zm\Zp\Zp\Zm\Zm\Zm\Zm\Zm\Zm\Zp\Zp\Zp\Zp\Zm\Zm\Zm\Zm\Zm\Zm\Zp\Zp\Zm\Zm\Zp\Zp\Zp\Zp\Zm\Zp\Zp\Zp\Zp\Zp\\
\Zp\Zp\Zm\Zp\Zp\Zm\Zm\Zp\Zp\Zm\Zp\Zm\Zp\Zm\Zm\Zp\Zm\Zp\Zp\Zm\Zp\Zm\Zp\Zm\Zm\Zp\Zp\Zm\Zm\Zp\Zm\Zp\Zm\Zm\Zm\Zp\\
\Zp\Zm\Zp\Zp\Zm\Zm\Zp\Zp\Zm\Zm\Zm\Zm\Zm\Zm\Zp\Zp\Zp\Zp\Zm\Zm\Zm\Zm\Zm\Zm\Zp\Zp\Zm\Zm\Zp\Zp\Zp\Zp\Zm\Zp\Zp\Zp\\
\Zp\Zp\Zm\Zp\Zm\Zp\Zp\Zm\Zm\Zp\Zp\Zm\Zp\Zm\Zp\Zm\Zm\Zp\Zm\Zp\Zp\Zm\Zp\Zm\Zp\Zm\Zm\Zp\Zp\Zm\Zm\Zp\Zm\Zp\Zm\Zm\\
\Zp\Zm\Zp\Zp\Zp\Zp\Zm\Zm\Zp\Zp\Zm\Zm\Zm\Zm\Zm\Zm\Zp\Zp\Zp\Zp\Zm\Zm\Zm\Zm\Zm\Zm\Zp\Zp\Zm\Zm\Zp\Zp\Zp\Zp\Zm\Zp
\end{array}
\right)
$$
\end{table}

\begin{table}[htbp!]
\centering
\caption{A pair of unbiased Hadamard matrices of order 36: second matrix}
\label{mat36k}
$$
K=\left(
\begin{array}{c}
\Zp\Zp\Zp\Zm\Zp\Zm\Zm\Zp\Zp\Zm\Zm\Zp\Zp\Zm\Zp\Zm\Zm\Zm\Zp\Zm\Zm\Zp\Zp\Zp\Zp\Zp\Zp\Zp\Zp\Zp\Zm\Zp\Zp\Zm\Zm\Zm\\
\Zp\Zm\Zm\Zm\Zm\Zm\Zp\Zp\Zm\Zm\Zp\Zp\Zm\Zm\Zm\Zm\Zm\Zp\Zm\Zm\Zp\Zp\Zp\Zm\Zp\Zm\Zp\Zm\Zp\Zm\Zp\Zp\Zm\Zm\Zm\Zp\\
\Zp\Zp\Zp\Zm\Zm\Zp\Zp\Zm\Zm\Zp\Zp\Zm\Zm\Zp\Zp\Zm\Zp\Zp\Zm\Zm\Zm\Zp\Zp\Zp\Zm\Zp\Zm\Zp\Zp\Zp\Zm\Zp\Zm\Zm\Zp\Zp\\
\Zp\Zm\Zm\Zm\Zp\Zp\Zm\Zm\Zp\Zp\Zm\Zm\Zp\Zp\Zm\Zm\Zp\Zm\Zm\Zp\Zp\Zp\Zp\Zm\Zp\Zp\Zp\Zp\Zp\Zm\Zp\Zp\Zm\Zp\Zp\Zm\\
\Zp\Zp\Zp\Zm\Zm\Zp\Zm\Zp\Zp\Zm\Zp\Zm\Zp\Zp\Zp\Zp\Zm\Zp\Zm\Zm\Zp\Zm\Zp\Zm\Zp\Zm\Zm\Zm\Zm\Zp\Zp\Zp\Zp\Zp\Zp\Zm\\
\Zp\Zm\Zm\Zm\Zp\Zp\Zp\Zp\Zm\Zm\Zm\Zm\Zp\Zm\Zp\Zm\Zp\Zp\Zm\Zp\Zm\Zm\Zm\Zm\Zm\Zm\Zm\Zp\Zp\Zp\Zp\Zm\Zp\Zm\Zm\Zm\\
\Zp\Zp\Zp\Zm\Zm\Zp\Zp\Zm\Zp\Zm\Zm\Zm\Zp\Zm\Zm\Zp\Zp\Zp\Zp\Zp\Zp\Zp\Zp\Zp\Zm\Zp\Zp\Zm\Zm\Zm\Zp\Zm\Zp\Zm\Zm\Zp\\
\Zp\Zm\Zm\Zm\Zp\Zp\Zm\Zm\Zm\Zm\Zm\Zp\Zm\Zm\Zp\Zp\Zp\Zm\Zp\Zm\Zp\Zm\Zp\Zm\Zp\Zp\Zm\Zm\Zm\Zp\Zm\Zm\Zm\Zm\Zp\Zp\\
\Zp\Zp\Zp\Zm\Zp\Zm\Zp\Zp\Zm\Zp\Zm\Zp\Zm\Zp\Zm\Zp\Zp\Zp\Zp\Zm\Zp\Zm\Zm\Zm\Zm\Zp\Zp\Zp\Zm\Zp\Zp\Zp\Zm\Zp\Zm\Zm\\
\Zp\Zm\Zm\Zm\Zm\Zm\Zp\Zm\Zp\Zp\Zp\Zp\Zp\Zp\Zp\Zp\Zp\Zm\Zm\Zm\Zm\Zm\Zm\Zp\Zp\Zp\Zp\Zm\Zp\Zp\Zp\Zm\Zp\Zp\Zm\Zp\\
\Zp\Zp\Zm\Zp\Zm\Zp\Zm\Zm\Zm\Zm\Zm\Zp\Zm\Zp\Zm\Zm\Zp\Zp\Zm\Zm\Zm\Zm\Zp\Zp\Zp\Zm\Zp\Zp\Zm\Zm\Zm\Zm\Zp\Zp\Zm\Zm\\
\Zp\Zm\Zp\Zp\Zp\Zp\Zm\Zp\Zm\Zp\Zp\Zp\Zp\Zp\Zm\Zp\Zp\Zm\Zm\Zp\Zm\Zp\Zp\Zm\Zm\Zm\Zp\Zm\Zm\Zp\Zm\Zp\Zp\Zm\Zm\Zp\\
\Zp\Zp\Zm\Zp\Zp\Zm\Zp\Zp\Zm\Zm\Zp\Zm\Zm\Zm\Zp\Zp\Zp\Zm\Zm\Zp\Zp\Zm\Zp\Zp\Zm\Zp\Zp\Zm\Zp\Zm\Zm\Zp\Zp\Zp\Zp\Zm\\
\Zp\Zm\Zp\Zp\Zm\Zm\Zp\Zm\Zm\Zp\Zm\Zm\Zm\Zp\Zp\Zm\Zm\Zm\Zp\Zp\Zm\Zm\Zp\Zm\Zp\Zp\Zm\Zm\Zm\Zm\Zp\Zp\Zp\Zm\Zm\Zm\\
\Zp\Zp\Zm\Zp\Zp\Zm\Zm\Zm\Zp\Zp\Zp\Zp\Zm\Zm\Zp\Zm\Zm\Zp\Zm\Zp\Zp\Zp\Zp\Zm\Zm\Zp\Zm\Zp\Zm\Zp\Zp\Zm\Zp\Zp\Zm\Zp\\
\Zp\Zm\Zp\Zp\Zm\Zm\Zm\Zp\Zp\Zm\Zp\Zm\Zm\Zp\Zm\Zm\Zp\Zp\Zp\Zp\Zp\Zm\Zm\Zm\Zp\Zp\Zp\Zp\Zp\Zp\Zm\Zm\Zp\Zm\Zp\Zp\\
\Zp\Zp\Zp\Zm\Zp\Zm\Zp\Zp\Zp\Zp\Zm\Zm\Zm\Zm\Zm\Zm\Zp\Zm\Zm\Zm\Zm\Zp\Zm\Zm\Zp\Zm\Zm\Zm\Zm\Zm\Zm\Zm\Zp\Zp\Zp\Zp\\
\Zp\Zm\Zm\Zm\Zm\Zm\Zp\Zm\Zp\Zm\Zm\Zp\Zm\Zp\Zm\Zp\Zm\Zm\Zm\Zp\Zp\Zp\Zm\Zp\Zm\Zm\Zm\Zp\Zm\Zp\Zm\Zp\Zp\Zm\Zp\Zm\\
\Zp\Zp\Zm\Zp\Zp\Zp\Zp\Zm\Zm\Zp\Zp\Zm\Zp\Zm\Zm\Zm\Zm\Zm\Zp\Zm\Zp\Zm\Zm\Zp\Zp\Zm\Zp\Zp\Zm\Zp\Zp\Zp\Zp\Zm\Zp\Zp\\
\Zp\Zm\Zp\Zp\Zp\Zm\Zm\Zm\Zp\Zp\Zm\Zm\Zm\Zm\Zm\Zp\Zm\Zp\Zm\Zm\Zm\Zm\Zp\Zp\Zm\Zm\Zp\Zm\Zp\Zp\Zp\Zm\Zm\Zm\Zp\Zm\\
\Zp\Zp\Zm\Zp\Zp\Zp\Zp\Zm\Zp\Zm\Zm\Zm\Zm\Zp\Zp\Zp\Zm\Zp\Zp\Zp\Zm\Zp\Zm\Zm\Zp\Zm\Zp\Zm\Zp\Zp\Zm\Zp\Zm\Zp\Zm\Zp\\
\Zp\Zm\Zp\Zp\Zp\Zm\Zm\Zm\Zm\Zm\Zm\Zp\Zp\Zp\Zp\Zm\Zp\Zp\Zp\Zm\Zp\Zp\Zm\Zp\Zm\Zm\Zm\Zm\Zp\Zm\Zp\Zp\Zp\Zp\Zp\Zp\\
\Zp\Zp\Zm\Zp\Zm\Zm\Zm\Zp\Zm\Zm\Zm\Zm\Zp\Zp\Zp\Zp\Zm\Zm\Zm\Zm\Zm\Zp\Zm\Zm\Zm\Zp\Zp\Zp\Zm\Zm\Zp\Zm\Zm\Zm\Zp\Zp\\
\Zp\Zm\Zp\Zp\Zm\Zp\Zp\Zp\Zm\Zp\Zm\Zp\Zp\Zm\Zp\Zm\Zm\Zp\Zm\Zp\Zp\Zp\Zm\Zp\Zp\Zp\Zp\Zm\Zm\Zp\Zm\Zm\Zm\Zp\Zp\Zm\\
\Zp\Zp\Zp\Zp\Zm\Zp\Zm\Zp\Zp\Zp\Zm\Zp\Zm\Zm\Zp\Zp\Zp\Zm\Zm\Zp\Zp\Zm\Zm\Zp\Zp\Zm\Zm\Zp\Zp\Zm\Zp\Zp\Zm\Zm\Zm\Zp\\
\Zp\Zm\Zp\Zm\Zp\Zp\Zp\Zp\Zp\Zm\Zp\Zp\Zm\Zp\Zp\Zm\Zm\Zm\Zp\Zp\Zm\Zm\Zp\Zp\Zm\Zm\Zp\Zp\Zm\Zm\Zp\Zm\Zm\Zp\Zp\Zp\\
\Zp\Zp\Zm\Zm\Zp\Zp\Zm\Zp\Zm\Zp\Zp\Zp\Zm\Zp\Zm\Zp\Zm\Zp\Zp\Zp\Zm\Zp\Zm\Zp\Zp\Zp\Zm\Zm\Zp\Zm\Zp\Zm\Zp\Zm\Zp\Zm\\
\Zp\Zm\Zm\Zp\Zp\Zm\Zp\Zp\Zp\Zp\Zp\Zm\Zp\Zp\Zp\Zp\Zp\Zp\Zp\Zm\Zp\Zp\Zp\Zp\Zp\Zm\Zm\Zp\Zm\Zm\Zm\Zm\Zm\Zm\Zm\Zm\\
\Zp\Zp\Zp\Zp\Zp\Zp\Zp\Zm\Zp\Zm\Zp\Zp\Zp\Zp\Zm\Zm\Zm\Zm\Zm\Zm\Zp\Zm\Zm\Zm\Zm\Zp\Zm\Zm\Zp\Zm\Zm\Zm\Zm\Zm\Zm\Zm\\
\Zp\Zm\Zp\Zm\Zp\Zm\Zm\Zm\Zm\Zm\Zp\Zm\Zp\Zm\Zm\Zp\Zm\Zp\Zm\Zp\Zm\Zm\Zm\Zp\Zp\Zp\Zm\Zp\Zm\Zm\Zm\Zp\Zm\Zp\Zm\Zp\\
\Zp\Zp\Zm\Zm\Zm\Zm\Zm\Zp\Zm\Zp\Zm\Zm\Zp\Zp\Zm\Zm\Zm\Zm\Zp\Zp\Zp\Zm\Zp\Zp\Zm\Zm\Zm\Zm\Zp\Zp\Zm\Zm\Zm\Zp\Zm\Zp\\
\Zp\Zm\Zm\Zp\Zm\Zp\Zp\Zp\Zp\Zp\Zm\Zp\Zp\Zm\Zm\Zp\Zm\Zp\Zp\Zm\Zm\Zm\Zp\Zm\Zm\Zp\Zm\Zp\Zp\Zm\Zm\Zp\Zp\Zp\Zp\Zp\\
\Zp\Zp\Zp\Zp\Zm\Zm\Zp\Zm\Zm\Zm\Zp\Zp\Zp\Zm\Zm\Zp\Zp\Zm\Zp\Zp\Zm\Zp\Zp\Zm\Zp\Zm\Zm\Zp\Zp\Zp\Zp\Zm\Zm\Zp\Zp\Zm\\
\Zp\Zm\Zp\Zm\Zm\Zp\Zm\Zm\Zm\Zp\Zp\Zm\Zm\Zm\Zp\Zp\Zm\Zm\Zp\Zm\Zp\Zp\Zm\Zm\Zm\Zm\Zp\Zp\Zp\Zm\Zm\Zm\Zp\Zp\Zm\Zm\\
\Zp\Zp\Zm\Zm\Zm\Zm\Zm\Zm\Zp\Zp\Zp\Zp\Zp\Zm\Zp\Zm\Zp\Zp\Zp\Zp\Zm\Zm\Zm\Zm\Zm\Zm\Zp\Zm\Zm\Zm\Zm\Zp\Zm\Zm\Zp\Zm\\
\Zp\Zm\Zm\Zp\Zm\Zp\Zm\Zp\Zp\Zm\Zp\Zm\Zm\Zm\Zm\Zm\Zp\Zm\Zp\Zm\Zm\Zp\Zm\Zp\Zm\Zp\Zm\Zm\Zm\Zp\Zp\Zp\Zm\Zp\Zm\Zm
\end{array}
\right)
$$
\end{table}

\begin{corollary}
There is a pair of weakly unbiased Hadamard matrices of order $20$.
\end{corollary}

\begin{proof}
We apply Theorem \ref {theorem:existenceWUHM} to the pair
of unbiased complex Hadamard matrices of order 10
of Table \ref{matC10}.
The resulting pair of matrices is given in Table \ref{mat20}.
All entries of $HK^*$ are in
$G(a,b)$, where $\{a,b\}=\{1,3\}$.
\end{proof}

\begin{table}[t!]
\centering
\caption{A pair $H$, $K$ of weakly unbiased Hadamard matrices of order 20}
\label{mat20}
$$
\left(
\begin{array}{c}
\Zp\Zp\Zp\Zp\Zp\Zp\Zp\Zp\Zp\Zp\Zp\Zp\Zp\Zp\Zp\Zp\Zp\Zp\Zp\Zp\\
\Zp\Zm\Zp\Zm\Zp\Zm\Zp\Zm\Zp\Zm\Zp\Zm\Zp\Zm\Zp\Zm\Zp\Zm\Zp\Zm\\
\Zp\Zp\Zm\Zm\Zm\Zp\Zm\Zp\Zp\Zm\Zm\Zp\Zp\Zm\Zp\Zm\Zp\Zm\Zm\Zp\\
\Zp\Zm\Zm\Zp\Zp\Zp\Zp\Zp\Zm\Zm\Zp\Zp\Zm\Zm\Zm\Zm\Zm\Zm\Zp\Zp\\
\Zp\Zp\Zm\Zp\Zm\Zm\Zm\Zp\Zp\Zm\Zp\Zm\Zm\Zp\Zm\Zp\Zp\Zm\Zp\Zm\\
\Zp\Zm\Zp\Zp\Zm\Zp\Zp\Zp\Zm\Zm\Zm\Zm\Zp\Zp\Zp\Zp\Zm\Zm\Zm\Zm\\
\Zp\Zp\Zm\Zp\Zm\Zp\Zm\Zm\Zm\Zp\Zp\Zm\Zp\Zm\Zp\Zm\Zm\Zp\Zp\Zm\\
\Zp\Zm\Zp\Zp\Zp\Zp\Zm\Zp\Zp\Zp\Zm\Zm\Zm\Zm\Zm\Zm\Zp\Zp\Zm\Zm\\
\Zp\Zp\Zp\Zm\Zp\Zm\Zm\Zp\Zm\Zm\Zm\Zp\Zm\Zp\Zp\Zm\Zm\Zp\Zp\Zm\\
\Zp\Zm\Zm\Zm\Zm\Zm\Zp\Zp\Zm\Zp\Zp\Zp\Zp\Zp\Zm\Zm\Zp\Zp\Zm\Zm\\
\Zp\Zp\Zm\Zp\Zp\Zm\Zp\Zm\Zm\Zp\Zm\Zm\Zm\Zp\Zp\Zm\Zp\Zm\Zm\Zp\\
\Zp\Zm\Zp\Zp\Zm\Zm\Zm\Zm\Zp\Zp\Zm\Zp\Zp\Zp\Zm\Zm\Zm\Zm\Zp\Zp\\
\Zp\Zp\Zp\Zm\Zm\Zp\Zp\Zm\Zm\Zp\Zm\Zp\Zm\Zm\Zm\Zp\Zp\Zm\Zp\Zm\\
\Zp\Zm\Zm\Zm\Zp\Zp\Zm\Zm\Zp\Zp\Zp\Zp\Zm\Zp\Zp\Zp\Zm\Zm\Zm\Zm\\
\Zp\Zp\Zp\Zm\Zm\Zp\Zp\Zm\Zp\Zm\Zp\Zm\Zm\Zp\Zm\Zm\Zm\Zp\Zm\Zp\\
\Zp\Zm\Zm\Zm\Zp\Zp\Zm\Zm\Zm\Zm\Zm\Zm\Zp\Zp\Zm\Zp\Zp\Zp\Zp\Zp\\
\Zp\Zp\Zp\Zm\Zp\Zm\Zm\Zp\Zm\Zp\Zp\Zm\Zp\Zm\Zm\Zp\Zm\Zm\Zm\Zp\\
\Zp\Zm\Zm\Zm\Zm\Zm\Zp\Zp\Zp\Zp\Zm\Zm\Zm\Zm\Zp\Zp\Zm\Zp\Zp\Zp\\
\Zp\Zp\Zm\Zp\Zp\Zm\Zp\Zm\Zp\Zm\Zm\Zp\Zp\Zm\Zm\Zp\Zm\Zp\Zm\Zm\\
\Zp\Zm\Zp\Zp\Zm\Zm\Zm\Zm\Zm\Zm\Zp\Zp\Zm\Zm\Zp\Zp\Zp\Zp\Zm\Zp
\end{array}
\right)
, \quad
\left(
\begin{array}{c}
\Zp\Zp\Zp\Zm\Zm\Zp\Zm\Zp\Zm\Zp\Zp\Zp\Zm\Zm\Zm\Zm\Zm\Zp\Zp\Zp\\
\Zp\Zm\Zm\Zm\Zp\Zp\Zp\Zp\Zp\Zp\Zp\Zm\Zm\Zp\Zm\Zp\Zp\Zp\Zp\Zm\\
\Zp\Zp\Zp\Zm\Zp\Zp\Zp\Zm\Zp\Zm\Zm\Zp\Zm\Zm\Zp\Zm\Zp\Zp\Zm\Zm\\
\Zp\Zm\Zm\Zm\Zp\Zm\Zm\Zm\Zm\Zm\Zp\Zp\Zm\Zp\Zm\Zm\Zp\Zm\Zm\Zp\\
\Zp\Zp\Zm\Zp\Zm\Zm\Zp\Zm\Zm\Zm\Zm\Zm\Zm\Zp\Zm\Zm\Zm\Zp\Zp\Zm\\
\Zp\Zm\Zp\Zp\Zm\Zp\Zm\Zm\Zm\Zp\Zm\Zp\Zp\Zp\Zm\Zp\Zp\Zp\Zm\Zm\\
\Zp\Zp\Zm\Zp\Zp\Zp\Zp\Zp\Zm\Zp\Zp\Zp\Zp\Zp\Zp\Zm\Zm\Zm\Zm\Zm\\
\Zp\Zm\Zp\Zp\Zp\Zm\Zp\Zm\Zp\Zp\Zp\Zm\Zp\Zm\Zm\Zm\Zm\Zp\Zm\Zp\\
\Zp\Zp\Zm\Zp\Zm\Zm\Zp\Zp\Zp\Zp\Zm\Zp\Zm\Zm\Zm\Zp\Zp\Zm\Zm\Zp\\
\Zp\Zm\Zp\Zp\Zm\Zp\Zp\Zm\Zp\Zm\Zp\Zp\Zm\Zp\Zp\Zp\Zm\Zm\Zp\Zp\\
\Zp\Zp\Zp\Zm\Zm\Zm\Zp\Zp\Zm\Zm\Zp\Zm\Zp\Zp\Zp\Zp\Zp\Zp\Zm\Zp\\
\Zp\Zm\Zm\Zm\Zm\Zp\Zp\Zm\Zm\Zp\Zm\Zm\Zp\Zm\Zp\Zm\Zp\Zm\Zp\Zp\\
\Zp\Zp\Zm\Zm\Zm\Zp\Zm\Zm\Zp\Zm\Zp\Zm\Zp\Zm\Zm\Zp\Zm\Zm\Zm\Zm\\
\Zp\Zm\Zm\Zp\Zp\Zp\Zm\Zp\Zm\Zm\Zm\Zm\Zm\Zm\Zp\Zp\Zm\Zp\Zm\Zp\\
\Zp\Zp\Zp\Zp\Zp\Zm\Zm\Zm\Zm\Zp\Zp\Zm\Zm\Zm\Zp\Zp\Zp\Zm\Zp\Zm\\
\Zp\Zm\Zp\Zm\Zm\Zm\Zm\Zp\Zp\Zp\Zm\Zm\Zm\Zp\Zp\Zm\Zm\Zm\Zm\Zm\\
\Zp\Zp\Zm\Zm\Zp\Zm\Zm\Zm\Zp\Zp\Zm\Zp\Zp\Zp\Zp\Zp\Zm\Zp\Zp\Zp\\
\Zp\Zm\Zm\Zp\Zm\Zm\Zm\Zp\Zp\Zm\Zp\Zp\Zp\Zm\Zp\Zm\Zp\Zp\Zp\Zm\\
\Zp\Zp\Zp\Zp\Zp\Zp\Zm\Zp\Zp\Zm\Zm\Zm\Zp\Zp\Zm\Zm\Zp\Zm\Zp\Zp\\
\Zp\Zm\Zp\Zm\Zp\Zm\Zp\Zp\Zm\Zm\Zm\Zp\Zp\Zm\Zm\Zp\Zm\Zm\Zp\Zm
\end{array}
\right)
$$
\end{table}

Consider the even integer $2n$, $n=a^2$ for some odd integer $a$,
and assume that $2n=a^2+a^2$ is the only way that $2n$ can be
written as sum of two squares. Let $H,K$ be two unbiased complex
Hadamard matrices $H,K$ of order $2n$. It is easy to see that
$HK^*=(a+ia)L$, where $L$ is a complex Hadamard matrix of order
$2n$. A pair of unbiased complex Hadamard matrices $H,K$ of order
$2n$, $2n=a^2+b^2$, is called  {\it special} if $HK^*=(a+ib)L$ for
some complex Hadamard matrix $L$. The unbiased complex Hadamard
matrices of orders 2 and 18 above are special. We did an
exhaustive computer search and found none of order 10.

\section{Unbiased bases}

Let $H,K$ be a pair of special unbiased complex Hadamard
matrices of order $2n^2$ corresponding to the decomposition
$2n^2=n^2+n^2$. Then the normalized rows of $H$ and $K$, or
equivalently the rows of
 $\frac{1}{{\sqrt {2n^2}}}H$ and $\frac{1}{{\sqrt {2n^2}}}K$, form two orthonormal bases
for $\mathbb{C}^{2n^2}$ in such a way that for every pair of
vectors $u,v$ from different bases, $\langle u,v\rangle\in
\Cset=\{{\frac{1}{2n}}(1+i), -{\frac{1}{2n}}(1+i),
{\frac{1}{2n}}(1-i), -{\frac{1}{2n}}(1-i)\}$ (note that $\frac{n+in}{2n^2}={\frac{1}{2n}}(1+i)$). Here $\langle, \rangle$ denotes the standard
Hermitian inner product in $\mathbb{C}^{2n^2}$. Adding
$\{\frac{1+i}{\sqrt {2}}~{\bf b}: ~{\bf b}\in B_s\}$, where $B_s$
denotes the standard basis in $\mathbb{C}^{2n^2}$, to these  bases
we get 3 orthonormal bases for $\mathbb{C}^{2n^2}$ in such a way
that for every pair of vectors $u,v$ from different bases,
$\langle u,v\rangle\in \Cset$.
 Two
orthonormal bases $\B_1$ and $\B_2$ in $\mathbb{C}^{2n^2}$ are
called {\it  unbiased complex bases} if $\langle u,v\rangle \in
\Cset$ for all $u\in \B_1$ and $v\in \B_2$.

We will use $|{\bf MUCB}(n)|$ to denote the number of elements in
a set of mutually unbiased complex bases for $\mathbb{C}^n$.


\begin{lemma} $|{\bf MUCB}(2n^2)|\le 3$
for any odd integer $n$.
Equality is attained for $n=1,3$.
\end{lemma}
\begin{proof}
Let $\B_1, \B_2,\B_3$ be three mutually unbiased complex bases for
$\mathbb{C}^{2n^2}$. Let $H_i$ be the matrix formed by putting the
vectors of $\B_i$ as the rows of $H_i$, $i=1,2,3$. Then
$\frac{2n}{1+i}H_2H_1^*$ and $\frac{2n}{1+i}H_3H_1^*$ form a pair
of unbiased complex Hadamard matrices of order $2n^2$. Thus, it
follows from Theorem \ref{theorem:numberMUCHtwo} that $|{\bf
MUCB}(2n^2)|-1\le 2$. The equality occurs for $n=1,3$ as there are
pair of  special unbiased complex Hadamard matrices of order $2$
and $18$.
\end{proof}
Two orthonormal bases ${\B}_1$ and $\B_2$ for ${\mathbb{R}}^n$ are
called {\it mutually unbiased real bases} if\\ $\langle
u,v\rangle \in\{\frac{1}{\sqrt n},-\frac{1}{\sqrt n}\}$ for all $u\in {\B}_1$ and $v\in
{\B}_2$, where $\langle, \rangle$ is the standard Euclidean inner
product in $\mathbb{R}^n$, see \cite{bstw} for details. We will
use $|{\bf MURB}(n)|$ to denote the number of elements in a set of
mutually unbiased real bases in ${\mathbb{R}}^n$.


\begin{lemma} $|{\bf MURB}(4n^2)|\le 3$ for any odd integer $n$. Equality is
attained for $n=1,3$.
\end{lemma}
\begin{proof}
Let $\B_1, \B_2,\B_3$ be three mutually unbiased real bases for
$\mathbb{R}^{4n^2}$. Let $H_i$ be the matrix formed by putting the
vectors of $\B_i$ as the rows of $H_i$, $i=1,2,3$. Then
$2nH_2H_1^t$ and $2nH_3H_1^t$ form a pair
of unbiased Hadamard matrices of order $4n^2$. The result now follows from Lemma \ref{w(n)} and Corollary \ref{36}. See also Observation 2.1 of
\cite {bstw}.
\end{proof}

{\bf Acknowledgments:} The authors wish to acknowledge, with appreciation, two long and detailed reports by an anonymous referee which
improved the presentation of this paper considerably. Thanks are also extended to
Professor Holzmann for his help.

\end{document}